\documentclass[14pt]{amsart}

\usepackage{amssymb}
\usepackage{amstext}
\usepackage{amsmath}
\usepackage{amscd}
\usepackage{latexsym}
\usepackage{amsfonts}
\usepackage{comment}
\usepackage{color}
\usepackage{enumerate}
\usepackage[all]{xy}
\usepackage{graphicx}

\theoremstyle{plain}
\newtheorem{thm}{Theorem}[section]
\newtheorem*{thm*}{Theorem}
\newtheorem*{cor*}{Corollary}
\newtheorem*{defn*}{Definition}

\newtheorem{prop}[thm]{Proposition}
\newtheorem{lem}[thm]{Lemma}
\newtheorem{que}[thm]{Quesion}
\newtheorem{cor}[thm]{Corollary}
\newtheorem{claim}[thm]{Claim}
\newtheorem*{claim*}{Claim}

\usepackage[pagebackref]{hyperref}
\usepackage{hyperref}
\hypersetup{
	colorlinks=true, 
	linkcolor=red, 
}

\usepackage[alphabetic,initials]{amsrefs}

\theoremstyle{definition}
\newtheorem{defn}[thm]{Definition}
\newtheorem{ex}[thm]{Example}
\newtheorem{rem}[thm]{Remark}

\theoremstyle{remark}

\newcommand{\rmg}{\mathrm{g}}

\newcommand{\calD}{\mathcal{D}}

\newcommand{\fka}{\mathfrak{a}}
\newcommand{\fkb}{\mathfrak{b}}

\newcommand{\fkm}{\mathfrak{m}}

\newcommand{\fkp}{\mathfrak{p}}
\newcommand{\fkq}{\mathfrak{q}}

\newcommand{\fku}{\mathfrak{u}}

\def\Ann{\mathrm{Ann }}
\def\Ass{\mathrm{Ass }}
\def\Assh{\mathrm{Assh }}

\def\e{\mathrm{e}}
\def\ir{\mathrm{ir }}

\def\m{\fkm}
\def\Hom{\mathrm{Hom}}

\usepackage{array}

\begin{document}
	
	\title{  On  the set of Chern numbers  in local rings}

	\author[H. L. Truong]{Hoang Le Truong}
	\address{Institute of Mathematics, VAST, 18 Hoang Quoc Viet Road, 10307
		Hanoi, Viet Nam} 
	\address{	Thang Long Institute of Mathematics and Applied Sciences, Hanoi, Vietnam}
	\email{hltruong@math.ac.vn\\
		truonghoangle@gmail.com}
	
	\author[H. N.Yen]{Hoang Ngoc Yen}
	%\address{Institute of Mathematics, VAST, 18 Hoang Quoc Viet Road, 10307
	%	Hanoi, Viet Nam}
	\address{The Department of Mathematics, Thai Nguyen University of education.
		20 Luong Ngoc Quyen Street, Thai Nguyen City, Thai Nguyen Province, Viet Nam.}
	\email{yenhn@tnue.edu.vn}
\thanks{The authors would like to thank K. Ozeki for the valuable comments to improve this article.}	
	
	\thanks{{\it Key words and phrases:} Gorenstein, Cohen-Macaulay, generalized Cohen-Macaulay, Chern numbers.
		\endgraf
		{\it 2020 Mathematics Subject Classification:}
		13H10, 13B22, 13H15 , Secondary 13D45.\\
		The last author was partially supported by  Grant number  ICRTM02-2020.05, awarded in the internal grant competition of International Center for Research and Postgraduate Training in Mathematics, Hanoi.
	}

	\date{}
	
	\maketitle

	\begin{abstract}
	
	%The purpose of this paper is to study the problem of when the Hilbert coefficients of certain $\fkm$-primary  ideals in a Noetherian local ring have uniform bounds, and when this is the case, to ask for their sharp bounds.
	
	%Let A be a Noetherian local ring with d = dim A > 0. This paper shows that the Hilbert coefficients {eiQ(A)}1≤i≤d of parameter ideals Q have uniform bounds if and only if A is a generalized Cohen-Macaulay ring.
	
	%This paper proposes to characterize the  Gorensteinness, Cohen-Macaulayness, generalized Cohen-Macaulaness of Noetherian local rings $(R,\fkm)$ in term of the behavior of the set of the Hilbert coefficients of certain $\fkm$ primary ideals in $R$.

%	of positive dimension such that
%	 the Hilbert coefficients of certain $\fkm$-primary ideals in $R$ range  among given values.
	
%	the first Hilbert coefficients of m-primary ideals in A range among only finitely many values

This paper purposes to characterize Noetherian local rings $(R, \fkm)$  such that the Chern numbers of certain $\fkm$-primary ideals in $R$   bounded above or range among only finitely many values. Consequently, we characterize the Gorensteinness, Cohen-Macaulayness, generalized Cohen-Macaulayness of local rings in terms of the behavior of its Chern numbers.	
	
	%In this paper, we explore the behavior of the set of the Hilbert coefficients of certain primary ideals in Noetherian local rings. Consequently, the main result of this study provides a characterization of a generalized Cohen-Macaulay ring in terms of its Hilbert coefficients  of non-parameter ideals. As corollaries to the main theorem, we obtain characterizations of a Gorenstein/Cohen-Macaulay ring in terms of its Chern coefficients of non-parameter ideals.

%		This paper mainly studies the behavior of the set of the Hilbert coefficients of certain primary ideals in  Noetherian local rings. The purpose is to characterize those local rings for which the Hilbert coefficients of certain primary ideals sit within given the range of values.
		
		%Let $(R,\m)$ be a Noetherian local ring with $d=\dim R >0$.
		%The purpose of this paper shows that 
		%the Hilbert coefficients $\{e_i (\fkq:_R \fkm, R)\}_{1\leq i \leq d}$ of the socle ideal  of parameter ideals $\fkq$ have uniform bounds if and only if $R$ is a generalized Cohen-Macaulay ring.

		%is to give bounds for the Chern numbers $e_1(\fkq:_R \fkm, R)$ and has applied to give characterize the  Gorensteinness, Cohen-Macaulayness of $R$. We also show that the Hilbert coefficients $\{e_i (\fkq:_R \fkm, R)\}_{1\leq i \leq d}$ have uniform bounds if and only if $R$ is a generalized Cohen-Macaulay ring.
	\end{abstract}

	\section{Introduction}

Let $(R,\fkm,k)$ be a Noetherian local ring of dimension $d$, where $\fkm$ is the maximal ideal and $k = R/\fkm$ is the residue field of $R$. %Assume that $R$ is a homomorphic image of a Cohen-Macaulay local ring. 
Let $M$ be a finitely generated $R$-module of dimension $s$.  
For an $\m$-primary ideal $I$ of $M$, it is well-known  that there are integers $\{e_i(I,M)\}_{i=0}^s$, called the {\it Hilbert coefficients} of $M$ with respect to $I$ such that for $n \gg 0$
\begin{eqnarray*}
	\ell_R(M/{I^{n+1}}M)={e}_0(I,M) \binom{n+s}{s}-{e}_1(I,M) \binom{n+s-1}{s-1}+\cdots+(-1)^s {e}_s(I,M),
\end{eqnarray*}
where $\ell_R(N)$ denotes the length of  an $R$-module $N$. In particular, the integer $e_0(I, M)>0$ is called the multiplicity
of $M$ with respect to $I$ and has been explored very intensively.
Notice that Nagata showed that $R$ is a regular local ring if and only if $e_0(\fkm,R) = 1$,
provided $R$ is unmixed (\cite{Sam51, Nag62}). %This was proven by P. Samuel \cite{Sam51} in the case where $R$ contains a field of characteristic $0$ and then by M. Nagata \cite{Nag55} in the above form. 
Recall that a local ring $R$ is unmixed, if $\dim R = \dim R/\fkp$ for every associated prime ideal $\fkp$ of the $\fkm$-adic completion $\hat{R}$ of $R$.
Moreover,  the Cohen-Macaulayness of $R$-module $M$ is characterized in terms of $e_0(\fkq,M)$ of parameter
ideals $\fkq$ of $M$. On the other hand, S. Goto and other authors in \cite{SMV10,GGH+10,GoO11,CGT13,GGH+14} analyzed
the behavior of the values $\{e_i(\fkq,M)\}_{i=1}^d$ for parameter ideals $\fkq$ of  modules $M$ and it is  used to characterize
Cohen-Macaulayness, Buchsbaumness, generalized Cohen-Macaulayness of  modules $M$.  In particularly, let %$\mathcal{P}(M)$ be the set of systems $\underline{x} = x_1, \ldots, x_s$ of parameters  $M$ and  
$$\Lambda_{i}(M) = \{e_i(\fkq,M) \mid \fkq \text{ is a parameter ideal of }  M \}, $$
for all $i = 1, \ldots, s$. Then we have the following results as in Table \ref{table1}.
\begin{table}[h!]
\centering
		\begin{tabular}{ | l | c | c| c|c|c|c|c|c|}
		\hline
	$\Lambda_{1}(M) \subseteq (- \infty, 0]$ & $M$ & \cite{SMV10,CGT13}  \\ \hline 
	$ 0 \in \Lambda_{1}(M)$,  ($\star$)&  $M$ is Cohen-Macaulay     & \cite{GGH+14}    \\ \hline
	$\left| \Lambda_{1}(M) \right| < \infty $,  ($\star$) &  $M$ is generalized Cohen-macaulay & \cite{GGH+14}    \\ \hline
	$\left| \Lambda_{i}(M) \right| < \infty$, for all $i = 1, \ldots, d$ & $R$ is generalized Cohen-macaulay  & \cite{GoO11}   \\ \hline
%	$\left| \Lambda_{1}(M) \right| = 1$,  ($\star$) & $M$ is Buchsbaum  &  \cite{GGH+14}  \\ \hline
%	$\left| \Lambda_{i}(M) \right| = 1 $, for all $i =1, \ldots, d$ & $R$ is Buchsbaum  &  \cite{GoO11}  \\ \hline
\end{tabular}
\caption{ Properties of a finitely generated module $M$ carried by the behavior of the certain set. A symbol ($\star$) requires that the module $M$ be unmixed. }
\label{table1}
\end{table}
The aim of our paper is to continue this research direction. Concretely, we will give characterizations
of  some special classes of rings in terms of its Hilbert coefficients with respect to certain non-parameter ideals.

%In the present paper we focus on the behavior of  the values of $\{e_i(\fkq :_R \fkm,  R)\}_{i=1}^d$ for parameter ideals $\fkq$ of module $R$. 

To state the results of this paper, first of all let us fix our notation and terminology.  Let $\fkb(M)=\bigcap\limits_{\underline{x},i=1}^{s}\Ann ((0):_{M/(x_1,\ldots,x_{i-1})M}x_i),$
where $\underline{x} = x_1,\ldots, x_s$ runs over all systems of parameters of $M$. A system $x_1, \ldots, x_s$ of parameters of $M$ is called a {\it $C$-system of parameters} of $M$ if $x_s \in \fkb(M)^3$ and $x_i \in \fkb(M/(x_{i+1}, \ldots, x_s)M)^3$ for all $i = s - 1, \ldots, 1$ and a parameter ideal $\fkq$ of $M$ is called {\it $C$-parameter ideal} if it is generated by a $C$-system of parameters of $M$ (\cite[Definition 2.15]{MoQ17}).
Notice that, $C$-systems of parameters of $M$ always exist, provided  $R$ is a homomorphic image of a Cohen-Macaulay local ring (see \cite{CuC18}). Moreover, the index of reducibility of  $C$-parameter ideals $\fkq$ of $M$, which  is  independent of the choice of  $\fkq$, is called the \textit{stable value} of $M$ and denoted by $\mathcal{N}_R(M)$ (see \cite{CuQ19}).
Let 
$$\Xi_{i}(M) = \{e_i(\fkq:_R\fkm,M) \mid \fkq \text{ is a }C\text{-parameter ideal of }  M \}, $$
for all $i = 1, \ldots, s$. Then the main results of this paper are  expressed as in the Table \ref{table2}.
\begin{table}[h!]
\centering
		\begin{tabular}{ | l | c | c| c|c|c|c|c|c|}
				\hline
		  $\Xi_{1}(R) \subseteq (- \infty, \mathcal{N}(R)]$,  ($\star$)& $R$ & Corollary   \ref{cor7}  \\ \hline 
		$ \mathcal{N}(R) \in \Xi_{1}(R)$ ,  ($\star$)& $R$   is Cohen-Macaulay     & Corollary   \ref{cor7}     \\ \hline
		$\left| \Xi_{1}(R) \right| < \infty $,  ($\star$) & $R$ is generalized Cohen-macaulay &  Theorem \ref{prop1}   \\ \hline
		$\left| \Xi_{i}(R) \right| < \infty$, for all $i = 1, \ldots, d$ & $R$ is generalized Cohen-macaulay  &Theorem \ref{thm2}   \\ \hline
%	5&	$\left| \Xi_{g_i}(R) \right| < \infty$, for all $i = \overline{2,d}$ and $ {\rm max } \, \Lambda_{g_1}(R) \leq \ell((0):_{H_{\fkm}^0(R)} \fkm)$ &  $R/H_{\fkm}^0(R)$   is Cohen-Macaulay  &   Theorem \ref{thm2} \\ \hline
\end{tabular}
\caption{  Properties of a finitely generated module $M$ carried by the behavior of the certain set. A symbol ($\star$) requires that the ring $R$ be unmixed.}
\label{table2}
\end{table}

This paper is divided into four sections.  In the next section we recall the notions and prove some
preliminary results on Hilbert coefficients, index of reducibility, $C$-system of parameters  and generalized Cohen-Macaulay modules. 
In the Section 3, we will explore the relation between the Chern numbers  and the stable value of $R$. This will apply in particular to characterize the Gorensteinness, Cohen-Macaulayness of local rings in term of  its Chern numbers   and the stable value (Theorem \ref{proe2c}, Corollary \ref{cor7} and Corollary \ref{cor3}).   
In the last section, we will study the problem of when the sets $\Xi_i(R)$ are finite, where $i = 1, \ldots, d$. 
We shall show that $R$ is generalized Cohen-Macaulay if and only if $\Xi_1(R)$ is finite, provided $R$ is unmixed (Theorem \ref{prop1}). 
Moreover  $R$ is generalized Cohen-Macaulay if and only if $\Xi_i(R)$  are finite for all $i = 1, \ldots, d$ (Theorem \ref{thm2}).
% Moreover, $R/H_{\fkm}^{0}$ is Cohen-Macaulay if and only if $\Xi_i(R)$  are finite, for all $i = 1, \ldots, d$ and $e_1(\fkq : \fkm) - e_1(\fkq) \leq r_d(R)$ for some $C$-parameter ideal of $R$ containing $\fkm^n$ ($n \gg 0$) (Theorem \ref{thm3}).

	\section{Preliminary}

	%In what follows, unless otherwise specified, 
	Throughout this paper, let $R$ denote a Noetherian local ring with maximal
ideal $\fkm$ of dimension $d > 0$. Assume that $R$ is a homomorphic image of a Cohen-Macaulay local ring. 
	Let $H^i_\fkm (\bullet)$ $(i \in \Bbb Z) $ be the $i$-th local cohomology functor of $R$ with respect
to $\fkm$. For each finitely generated $R$-module M let $\ell_R(M)$ stand for the length of $M$. Moreover, if the module
$H^i_\fkm (M)$ is finitely generated, its length is denoted by $h_i(M)$ and $r_i(M)=\ell((0):_{H^i_\fkm(M)}\fkm)$. 	We also put $\fka_i=\Ann\, H^j_\fkm(M)$ for all $j\in \mathbb{Z}$, and $\fka(M)=\fka_0(M)\ldots\fka_{s-1}(M)$. We denote $\fkb(M)=\bigcap\limits_{\underline{x},i=1}^{s}\Ann ((0):_{M/(x_1,\ldots,x_{i-1})M}x_i)$
	where $\underline{x} = x_1,\ldots, x_s$ runs over all systems of parameters of $M$.

 % and $x_1,\ldots,x_s$ a system of parameter of $M$.   We denote by $\fkq_i$ the ideal $(x_1, \ldots,x_i)R$ for $i = 1,\ldots ,s$ and stipulate that $\fkq_0$ is the zero ideal of $R$. 
	%Suppose $\bigcap_{\fkp \in \Ass M} N_{\fkp} = 0 $ is a reduced primary decomposition of the zero submodule of $M$, then the unmixed component of	$M$ is defined by	$$ U_M(0) =  \bigcap_{\fkp \in \Ass M, \dim R/\fkp = s} N_{\fkp}.$$
	%We put $r_j(M) = \ell(0:_{H_{\fkm}^{j}(M)} \fkm)$ for all $j\in\Bbb Z$.   
\subsection{Hilbert coefficients} 	It is well-known  that there are integers $\{\e_i(I,M)\}_{i=0}^s$, called the {\it Hilbert coefficients} of $M$ with respect to $I$ such that 
	\begin{eqnarray*}
		\ell_R(M/{I^{n+1}}M)={\e}_0(I,M) \binom{n+s}{s}-{\e}_1(I,M) \binom{n+s-1}{s-1}+\ldots+(-1)^s {\e}_s(I,M).
	\end{eqnarray*}
	for large integer $n$. In particular, the leading coefficient $\e_0(I)$ is said to be {\it the multiplicity} of $M$ with respect to $I$ and $\e_1(I)$, which Vasconselos (\cite{Vas08}) refers to as the {\it Chern number} of $M$ with respect to $I$. For inductive arguments that used in this paper we need the following results. 
	\begin{lem}[{\cite[22.6]{Nag62}}] \label{lmsup}
		For every superficial element $x \in I$, we have
		$$ e_{j}(I,M/xM) =
		\begin{cases} e_{j}(I,M)  &\text{if $0 \le j \le s-2$,}
			\\
			e_{s-1}(I,M) + (-1)^{s}\ell_{R}(0 :_M x) &\text{if $0 \le j = s-1$}.
		\end{cases}$$  
	\end{lem}
	\begin{lem}[{\cite[Lemma 3.4]{CGT13}}] \label{lm3.4}
		Let $N$ be a submodule of $M$ with $\dim N = t < s$. Then
		$$ e_{j}(I,M) =
		\begin{cases} e_{j}(I,M/N)  &\text{if $0 \le j \le s-t-1$,}
			\\
			e_{s-t}(I,M/N) + (-1)^{s-t} e_0(I, N) &\text{if } j = s -t.
		\end{cases}$$  
	\end{lem}

\subsection{Index of reducibility}	Recall that a proper submodule $N$ of $M$ is called {\it irreducible}  if $N$ can not be written as an  intersection of two strict larger submodules of $M$. For a submodule $N$ of $M$, the number of irreducible components of an irredundant irreducible decomposition of $N$, which is independent of the choice of the decomposition, is called the \textit{index of reducibility} of $N$ and denoted by  $\ir_M(N)$ (\cite{Noe21}).  
	%We denote by $\Soc(N)$ the socle of $N$. Since $\Soc(N) \cong 0 :_N \m \cong Hom(\mathfrak{k},N)$ is a $\mathfrak{k}$-vector space.
	For an $\fkm$-primary  ideal $I$, we have 
	$$\ir_M(I) : = \ir_M(I M) =\ell_R([I M :_M \fkm]/I M).$$
	%Similarly, for $I$ is  $\frak m$-primary ideal  of $M$, we also have
	%$$\ir_M(I) : = \ir_M(I M) =\ell_R([I M :_M \fkm]/I M).$$
	Moreover, 
	there exists  the integers $f_i(I,M)$ such that
	$$\mathrm{ir}_M(I^{n+1})=\sum\limits_{i=0}^{d-1}(-1)^if_i(I,M)\binom{n+d-1-i}{d-1-i}.$$
	for all large $n$  (\cite[Lemma 4.2]{CQT15}).
	%These integers  $f_i(I,M)$   are called the irreducible coefficients of $M$ with respect to $I$. 
	The leading coefficient $f_0(I,M)$ is  called {\it the irreducible multiplicity} of $M$ with respect to $I$.
	%When $M=R$, we abbreviate $f_0(I, M)$ to $f_0(I)$.

	A system $x_1, \ldots, x_s$ of parameters of $M$ is called a {\it $C$-system of parameters} of $M$ if $x_s \in \fkb(M)^3$ and $x_i \in \fkb(M/(x_{i+1}, \ldots, x_s)M)^3$ for all $i = s - 1, \ldots, 1$.  A parameter ideal $\fkq$ of $M$ is called {\it $C$-parameter ideal} if  		it is generated by a $C$-system of parameters of $M$ (\cite[Definition 2.15]{MoQ17}). Notice that, $C$-systems of parameters of $M$ always exist, provided  $R$ is a homomorphic image of a Cohen-Macaulay local ring (see \cite{CuC18}) and 
	a $C$-system of parameters forms a $d$-sequence. Moreover, the index of reducibility of  $C$-parameter ideals $\fkq$ of $M$, which  is  independent of the choice of  $\fkq$, is called the \textit{stable value} of $M$ and denoted by $\mathcal{N}_R(M)$ (see \cite{CuQ19}). In particular, $\mathcal{N}_R(M) =    r_s(M)$  provided $M$ is Cohen-Macaulay.

	Furthermore, we have the following results which are useful in this paper.
	
	\begin{lem} \label{lmc}
		Let $x_1, \ldots, x_s$ be a $C$-system of parameters of $M$. Then
		\begin{enumerate}[$i)$]
			\item $x^{n_1}_1, \ldots, x^{n_s}_s$ is a $C$-system of parameters of $M$ for all $n_1, \ldots, n_s \geq 1$.
			\item $x_1, \ldots, x_{i-1}, x_{i+1}, \ldots, x_s$ is a $C$-system of parameters of $M/x_iM$ and$$ \mathcal{N}_R(M) = \mathcal{N}_R(M/x_iM),$$
			for all $i=1,\ldots, s$.
			\item  $x_1, \ldots, x_s$ is a $C$-system of parameters of $M/W$ and
			$$ \mathcal{N}_R(M) = \mathcal{N}_R(M/W ) + r_0(M),$$
			where $W=H^0_\fkm(M)$.
		\end{enumerate}
		%(v) There exists a positive integer $n_0$ such that $x_1^n, x_2^n, \ldots, x_s^n$ forms a distinguished system of parameter of $M$ for all $n \geq n_0$. 
	\end{lem}
	\begin{proof}
		$i$) and $ii$) are followed from \cite[Lemma 2.13]{CuQ20}.\\
		$iii$), %Let $\fkq = (x_1, \ldots, x_d)$. Then  $ x_1, \ldots, x_s$ is  also a $C$-system of parameters of $M/W$ (\cite[Lemma 2.13]{CuQ20}).
		Since $W$ has finite length, by $i$) and Lemma 2.4 in \cite{CuT08} there exists a $C$-parameter ideal $\fkq$ of $M$ such that $\fkq M \cap W = 0$ and $ (\fkq M + W) :_M  \m = \fkq M:_M \m + W$. Therefore, we have the exact sequence 
		$$ 0 \to W \to M/\fkq M \to M/W + \fkq M \to 0.  $$
		After applying the functor $\Hom_R(\mathfrak{k}, \bullet)$ to the above exact sequence, the sequence 
		$$ 0 \to (0):_{W} \m \to (0):_{M/\fkq M} \m \to (0):_{M/W + \fkq M} \m \to 0$$
		is exact.
		Since $\fkq$ is also a $C$-parameter ideal of $M/W$, we have  
		\begin{eqnarray*}
			\mathcal{N}_R(M) &=& \ell((0):_{M/\fkq M} \fkm) \\
			&=& \ell ((0):_{M/W + \fkq M} \fkm) + \ell((0):_{W} \m) = \mathcal{N}_R(M/W) + r_0(M),
		\end{eqnarray*}	 
	this complete the proof.
		%(iii) We choice 
		%$$x_i \in \b(M/(x_{i+1, \ldots, x_s})M+U)^3 \cap \b(M/(x_{i+1, \ldots, x_s})M)^3$$ 
		%for all $i = 1, \ldots, s$ then $x_1, \ldots, x_s$ forms $C$-system of parameter of $M$ and $L$. Hence $\mathcal{N}_R(M) = \mathcal{N}_R(L) = \ir_M((x_1, \ldots, x_s))$.
		%(iii) Let $ x_1, \ldots, x_s$ be a $C$-system of parameters of $M$ and put $\fkq = ( x_1, \ldots, x_s)$. By Lemma \ref{lm2.13CQ20} (ii) and Fact \ref{fact2}, we have 
		%$$  \mathcal{N}_R(M) = \ir_{M}(\fkq) =\sum_{i=0}^{s}\binom{d}{i} r_i(M).$$
		%(iv) This is a direct consequence of (iii).
	\end{proof}
	The following results are useful for our inductive technique.
	\begin{lem}[{\cite[Remark 3.3]{MoQ17}}] \label{fact1}
		Assume that $M$ is unmixed. Let $ x_1, \ldots, x_s$ be a $C$-system of parameters of $M$. Then
		\begin{enumerate}[$i)$]
			\item   $H_{\fkm}^{1}(M)$  is finitely generated provided $s \ge 2$. 
			%\item[$(ii)$] The set
			%$$ \mathcal{F}(R) = \{\fkp \in {\rm Spec}(R) \mid \dim R_{\fkp} > 1 = {\rm depth} R_{\fkp}, \fkp \neq \fkm\}$$
			%is finite.
			\item  $\Ass (M/x_1M) \subseteq \Assh (M/x_1M) \cup \{\fkm\}$, where $\Assh M = \{\fkp \in \Ass M \mid \dim R/\fkp = \dim M \}$.
		\end{enumerate}
	\end{lem}
	
\subsection{Generalized Cohen-Macaulay}
	
	Now we recall the notions of generalized Cohen-Macaulay modules and standard
	parameter ideals in terms of local cohomology (\cite{CST78,Tru86}).
	\begin{defn}
		\begin{enumerate}[$i)$]
			\item An $R$-module $M$ is said to be  {\it generalized Cohen-Macaulay}, if $H_{\m}^i(M)$ are of finite length for all $i <s$. 
			\item  A parameter ideal $\fkq = (x_1,\ldots,x_s)$ of $M$ is called {\it standard} if $$\fkq H_{\m}^i (M/(x_1,\ldots,x_j)M) = 0,$$ for all $0 \leq i + j < s$.
		\end{enumerate}
	\end{defn}
	
	It is well known that if  $M$ is a generalized Cohen-Macaulay module, then there exists a positive integer $n$ such that every parameter ideals of $M$ contained in $\m^{n}$ is standard.	 Note that a standard system of parameters forms a $d$-sequence.   
	
		\begin{lem} \label{fact2c} \it 
		Let $\fkq$ be a $C$-parameter ideal of a generalized Cohen-Macaulay $R$-module $M$. Then we have
		%which satisfies $I^2 = \fkq I$, where $I = \fkq : \m$ and $$ \ir_{R}(\fkq) = \sum_{i=0}^{d}\binom{d}{i} r_i(R).$$
		
		\begin{enumerate}[$i)$]
		\item $\fkq$ is standard and $ \mathcal{N}_R(M) =  \sum_{i=0}^{s}\binom{s}{i} r_i(M).$

			%\item There exists a positive integer $n$ such that every parameter ideal of $M$ contained in $\m^{n}$ is standard
			%\item   $\fkq$ is standard parameter ideal {\rm (see \cite{Tru86})},
			%\item $I^2 = \fkq I$, where $I = \fkq : \m$ {\rm (see \cite[Theorem 1.2]{CuT08})},
			%\item$ \ir_{R}(\fkq) = \sum_{i=0}^{d}\binom{d}{i} r_i(R)$ {\rm (see \cite[Theorem 1.1]{CuT08})},
			\item $ f_0(\fkq, M) = \displaystyle\sum_{j=1}^{s }\binom{s-1}{j-1} r_j(M)$ {\rm (see  \cite[Lemma 4.2]{Tru17})}.
			%		\item 
			%$$  e_i(I, R) = \begin{cases}
			% (-1)^{i+1} \left(\sum_{j=1}^{d-i} \binom{d-i-1}{j-1} \ell_{R}(H_{\m}^{j}(R))   \right) &\text{ if } i=1,\ldots,d-1\\
			%\ell_{R}(H_{\m}^{0}(R)) &\text{ if } i=d
			%\end{cases}$$
			
			%\item In the case $M =R$, 
			%$ e_1(\fkq:\m, R) - e_1(\fkq, R) =  f_0(\fkq, R) = \sum_{j=1}^{d }\binom{d-1}{j-1} r_j(R).$
			\item In the case $M =R$, 
			$$  e_i(\fkq:_R\fkm, R)= \begin{cases}
				(-1)^i \left( \displaystyle\sum_{j=1}^{d-i} \binom{d-i-1}{j-1} \ell_{R}(H_{\m}^{j}(R)) - \sum_{j=1}^{d-i+1} r_j(R) \right)  &\text{ if } i=1,\ldots,d-1,\\
				(-1)^d( \ell_{R}(H_{\m}^{0}(R)) - r_1(R))&\text{ if } i=d,
			\end{cases}$$
			%	where $I = \fkq : \fkm$ 
			{\rm (see \cite[Theorem 5.2]{CQT19}\label{thm5.2cqt19}).}
			
		\end{enumerate}
	\end{lem}
	
	%Let $x_1, x_2,\ldots x_s$ be a $C$-system of parameters of $M$. So are $x_1^{n_1}, x_2^{n_2},\ldots x_s^{n_s}$ for all integers $n_j\ge 1$.
	
	%Recall  that $R$ is unmixed if $\Ass \hat{R}= \Assh \hat{R}$, where $\hat{R}$ denotes the $\m$-adic completion of $R$. If $R$ is a homomorphic image of a Cohen-Macaulay ring, then $R$ is unmixed if and only if $\Ass R = \Assh R$.

	%%%%%%%%%%%%%%%%%%%%%%%%%%%%%%%%%%%%%%%%%%%%%%%%%%%%%%%%%%%%%%%
	\section{Characterizations of Cohen-Macaulay rings}
	The purpose of this section is to present characterizations of  Cohen-Macaulay rings in terms of
	its Chern numbers, irreducible multiplicities, stable value and the Cohen-Macaulay type. As corollaries, we obtained the
	characterizations of a Gorenstein ring in terms of its Chern numbers, irreducible multiplicity and stable value. We begin with the following result.

	%%%%%%%%%%%%%%%%%%%%%%%%%%%%%%%%%%%%%%%%%%%%%%%

	%%%%%%%%%%%%%%%%%%%%%%%%%%%%%%%%%%%%%%%%%%%%%%%%%%%
	%\begin{prop}\label{proe1c}
	%	Assume that $M$ is unmixed with $\dim M \ge 2$. Then the following statements are equivalent.
	%	\begin{enumerate}[$i)$] 
	%\item $M$ is Cohen-Macaulay.\\
	%\item  For some $C$-parameter $\fkq$ of $M$, we have 
	%	$$ \mathcal{N}(M) \leq f_0(\fkq, M).$$
	%\end{prop}

	\begin{prop}\label{proe1c}
		Assume that $M$ is unmixed  of dimension $d \geq 2$. Suppose that 
		$$ \mathcal{N}(M) \leq f_0(\fkq, M),$$
		for some $C$-parameter ideal $\fkq$ of $M$. Then $M$ is Cohen-Macaulay. 
	\end{prop}
	\begin{proof}
		We use induction on the dimension $s$ of $M$. In the case $\dim M = 2$, then $M$ is  generalized Cohen-Macaulay because of Lemma \ref{fact1}. We get by the Lemma \ref{fact2c} that 
		$$r_1(M) + r_2(M) =  f_0(\fkq, M) \geq \mathcal{N}(M) = r_2(M) + 2 r_1(M).$$ 
		Thus, $r_1(M) = 0$ and therefore  $M$ is Cohen-Macaulay.

		Suppose that $\dim M \geq 3$ and that our assertion holds true for $\dim M - 1$. %Let $\fkq = (x_1, x_2, \ldots , x_d)$ be a $C$-parameters of $R$. 
		Since $M$ is unmixed, $x_1$ is $M$-regular whence  $H_{\fkm}^{0}(M)  =  (0) :_M x_1= (0)$. Let $A = M/x_1M$ and $\overline{\fkq} = (x_2, \ldots, x_s)$. 
		Then by Lemma \ref{fact1} and Lemma \ref{lmc} ii), we have  $\Ass A \subseteq \Assh A \cup \{ \fkm\}$ and $x_2, \ldots, x_d$ is a $C$-system of parameters of $A$. Thus $B := A/H_{\fkm}^{0}(A)$ is unmixed and  $x_2, \ldots, x_d$ is  also a $C$-system of parameters of $B$. 
		It follows from Lemma \ref{lmc} that we have
		$$  \mathcal{N}(B)\le \mathcal{N}(A) =\mathcal{N}(M).$$
		On the other hand, we have
		$$ f_0(\fkq, M) \leq  f_0(\overline{\fkq}, A) \leq f_0(\overline{\fkq}, B),$$
		because of Lemma 2.1 and 2.2 in \cite{TTr 20}. Consequently,  we get that $ \mathcal{N}(B) \leq f_0(\overline{\fkq}, B)$. 
		By the hypothesis of induction, $B$ is Cohen-Macaulay. Thus $H^{i}_{\m}(A) = 0$ for all $1 \leq i < d$.

		Now it follows from the  following sequence 
		$$0 \to M \xrightarrow{x_1} M \to A \to 0$$
		that we have the long exact sequence
		\begin{eqnarray*}
			\ldots \to H^{1}_{\m}(M) \xrightarrow{x_1} &H^{1}_{\m}(M)& \to H^{1}_{\m}(A)\to 0 \ldots \\
			\ldots \to H^{i}_{\m}(M) \xrightarrow{x_1} &H^{i}_{\m}(M)& \to H^{i}_{\m}(A)\to 0 \ldots 
		\end{eqnarray*} 
		Then we have $H^{i}_{\m}(M) = 0$ for all $2 \leq i \leq d-1$ and $H^{1}_{\m}(M) = x_1H^{1}_{\m}(M)$. Thus $H^{1}_{\m}(M) = 0$
		because $H^{1}_{\m}(M)$ is a finite generated $R$-module. Hence $M$ is Cohen-Macaulay, as required.
	\end{proof}
	
	%%%%%%%%%%%%%%%%%%%%%%%%%%%%%%%%%%%%%%%%%%%%%%%%%%%%%%%%%%%%%%
	
	For each integer $n\geq 1$, we denote by ${\underline x}^n$ the sequence $x^n_1, x^n_2,\ldots,x^n_d$. Let $K^{\bullet}(x^n)$ be the
	Koszul complex of $R$ generated by the sequence ${\underline x}^n$ and let
	$H^{\bullet}({\underline x}^n,R) = H^{\bullet}(\Hom_R(K^{\bullet}({\underline x}^n),R))$
	be the Koszul cohomology module of $R$. Then for every $p\in\Bbb Z$, the family $\{H^p({\underline x}^n,R)\}_{n\ge 1}$
	naturally forms an inductive system of $R$, whose limit
	$$H^p_\fkq(R)=\lim\limits_{n\to\infty} H^p({\underline x}^n,R)$$
	is isomorphic to the local cohomology module
	$$H^p_\fkm(R)=\lim\limits_{n\to\infty} {\rm Ext}_R^p(R/\fkm^n,R).$$
	For each $n\geq 1$ and $p \in\mathbb{Z}$, let $\phi^{p,n}_{{\underline x},R}:H^p({\underline x}^n,R)\to H^p_\fkm(R)$ denote the canonical
	homomorphism into the limit. 
	\begin{defn}[{\cite[Lemma 3.12]{GoS03}}]\label{sur}
		{\rm There exists an integer $n_0$ 
			such that for all systems of parameters ${\underline x}=x_1,\ldots,x_d$  for $R$ contained in $\fkm^{n_0}$ and for all $p\in \Bbb Z$, the canonical homomorphisms
			$$\phi^{p,1}_{{\underline x},R}:H^p({\underline x},R)\to H^p_\fkm(R)$$
			into the inductive limit are surjective on the socles.
			The least integer $n_0$ with this property is called a \textit{g-invariant} of $R$ and denote by $\rmg(R)$.}
	\end{defn}

	Now let $\bigcap_{\fkp \in \Ass M} N_{\fkp} = 0 $ be a reduced primary decomposition of the zero submodule of $M$. Then the submodule
	$ U_M(0) =  \bigcap_{\fkp \in \Ass M, \dim R/\fkp = s} N_{\fkp}$,
	is called the unmixed component of $M$ and denoted by $U_M(0)$. We denote $\fku$ the unmixed component of $R$. %and let $S=R/\fku$.

	%\begin{prop}\label{proe1c}
	%	Assume that $R$ is unmixed with $\dim R \ge 2$. Then the following statements are equivalent.
	%	\begin{enumerate}[$i)$] 
	%\item $R$ is Cohen-Macaulay.\\
	%\item  For some $C$-parameter $\fkq$ of $M$, we have 
	%	$$e_1(\fkq : \m, R) - e_1(q, R) \leq r_d(R).$$
	%\end{prop}
	%%%%%%%%%%%%%%%%%%%%%%%%%%%%%%%%%%%%%%%%%%%%%%%%%%%%%%%%%%%%
		
	\begin{defn}	[{\cite{Sch99}}]
		\begin{enumerate}[$i)$]
			\item  A filtration 
			$\calD : D_0=(0)\subseteq D_1\subsetneq D_2\subsetneq \cdots \subsetneq D_{t}=M$ of submodule of $M$ is said to be a \textit{dimension filtration}, if $D_i$ is the largest submodule of $D_{i+1}$ with $\dim D_i < \dim D_{i+1}$ for all $i=1, \ldots, t-1$. 
			\item  A  system $\underline{x} =x_1,x_2, \ldots, x_s$  of parameters of $M$ is called %{\it good} (resp. 
			{\it distinguished}, if  
			$  (x_j \mid d_{i} < j \le s) D_i=(0)$
			%$$(x_j \mid d_{i} < j \le d)M \cap D_i=(0) \ \ \text{(resp.}\ \  (x_j \mid d_{i} < j \le d) D_i=(0))$$ 
			for all $1\le i\le t$. A parameter ideal $\fkq$ of $M$ is called  %{\it good} (resp. 
			{\it distinguished},  if  it is generated by a distinguished system of parameters of $M$. Note that every $C$-system of parameter is distinguished, (see \cite[Proposition 4.8 and Remark 4.10]{CuQ17}).
			
		\end{enumerate}	
		
	\end{defn}	
	%%%%%%%%%%%%%%%%%%%%%%%%%%%%%%%%%%%%%%%%%%%%%%%%%%%%%%%%%%%%%%5
	\begin{lem} [{\cite[Proposition 3.11]{OTY20} \label{prof0c}}]
		Suppose that 
		$$e_1(\fkq : \m, R) - e_1(\fkq, R) \leq r_d(R)$$
		for some $C$-parameter ideal $\fkq\subseteq \fkm^{\rmg(R)}$ of $R$.  Then $R/\fku$ is Cohen-Macaulay. 
	\end{lem}

	%%%%%%%%%%%%%%%%%%%%%%%%%%%%%%%%%%%%%%%%%%%%%%%%%%%%%%%%%%%
	\begin{lem}\label{5.1}
		Assume that $\dim R \geq 2$. Then  we have
		$$e_1(\fkq:_R\fkm,R)-e_1(\fkq,R)\le f_0(R),$$
		for all $C$-parameter ideals $\fkq$.
	\end{lem}
	\begin{proof} Let $I=\fkq:_R\fkm$. 
		In the case that	 $e_0(\fkm,R)=1$. It follows from  $R$ is unmixed and Theorem 40.6 in \cite{Nag62} that $R$ is Cohen-Macaulay. We get by Lemma \ref{fact2c}  that 
		$$e_1(I,R)-e_1(\fkq,R) = f_0(\fkq,R).$$
		
		Now suppose that $e_0(\m,R) > 1$.   Since $\dim R \ge 2$, by Proposition 2.3 in \cite{GoS03}, we have $\m I^n = \m\fkq^n$ for all $n$. Therefore $I^n\subseteq \fkq^n:\m$ for all $n$. Consequence, we obtain
		$$\ell(R/\fkq^{n+1})-\ell(R/I^{n+1})=\ell(I^{n+1}/\fkq^{n+1})\le \ell((\fkq^{n+1}:\m/\fkq^{n+1}).$$ 
		Hence we have $$e_1(I,R)-e_1(\fkq,R)\le f_0(\fkq,R),$$
		as required. 
	\end{proof}

In \cite{Tru14} and \cite{Tru17}, the first author provided the characterizations of  Cohen-Macaulay rings in terms of
its Chern numbers, irreducible multiplicities, index of reducibility of a parameter of $R$, and the Cohen-Macaulay type, provided $R$ is unmixed. Notice that  the necessary and sufficient conditions of these characterizations need to hold true for all parameter ideals of $R$. Recently,  N. T. T. Tam and the first author in \cite{TTr20} gave the characterizations of  Cohen-Macaulay rings in terms of
its Chern numbers, irreducible multiplicities and the type, 
which was introduced  by S. Goto and N. Suzuki (\cite{GoS84}).	Let us now state the main results of this section and its corollaries.
%	We now state the main result of this section.
	\begin{thm}\label{proe2c}
		Assume that $R$ is unmixed of dimension $d \geq 2$. Then the following statements are equivalent.
		\begin{enumerate}[$(1)$] 
			\item   $R$ is Cohen-Macaulay.
			%\item There exists a positive integer $n$ such that for all parameter ideals $\fkq \subseteq \m^n$, one has
			%$$ r_d(R) = e_1(\fkq:\m) - e_1(\fkq) = f_0(\fkq) = %\mathcal{N}(R),$$
			%which is independent of $\fkq$.
			\item For some $C$-parameter ideal $\fkq$ of $R$, we have
			$$ \mathcal{N}(R) \leq e_1(\fkq:\m) - e_1(\fkq).$$
			\item For some $C$-parameter ideal $\fkq$ of $R$ , we have
			$$ \mathcal{N}(R) \leq f_0(\fkq, R).$$
			\item For some $C$-parameter ideal $\fkq\subset \m^{\rmg(R)}$ of $R$, we have 
			$$ f_0(\fkq)  \leq r_d(R).$$
			\item For some $C$-parameter ideal $\fkq\subset\m^{\rmg(R)} $ of $R$, we have 
			$$  e_1(\fkq:\m) - e_1(\fkq)\leq r_d(R).$$
		\end{enumerate}
	\end{thm}
	\begin{proof}
		$(1) \Rightarrow (2)$ and $(1) \Rightarrow (4) $ follow from Lemma \ref{fact2c}.\\
		$(2) \Rightarrow (3)$ and $(4) \Rightarrow (5) $ are trivial.\\
		$(3) \Rightarrow (1) $ and $ (5) \Rightarrow (1)$ are   immediate from Proposition  \ref{proe1c} and Lemma \ref{prof0c}.
	\end{proof}
	%%%%%%%%%%%%%%%%%%%%%%%%%%%%%%%%%%%%%%%%%%%%%%%%%%

Let 
$$\Xi_{i}(M) = \{e_i(\fkq:_R\fkm,M) \mid \fkq \text{ is a }C\text{-parameter ideal of }  M \}, $$
for all $i = 1, \ldots, s$. Then we have the following results.

%we get by Theorem \ref{proe2c} that $\Xi_{1}(R)$ bounded above by $\mathcal{N}(R)$ and $\mathcal{N}(R) \in \Xi_{1}(R)$ if and only if $R$ is Cohen-Macaulay, provided R is unmixed.  

	\begin{cor} \label{cor7}
		Assume that $R$ is unmixed of dimension $d \geq 2$. Then $$\Xi_{1}(R) \subseteq (- \infty, \mathcal{N}(R)].$$ Moreover, $\mathcal{N}(R) \in \Xi_{1}(R)$ if and only if $R$ is 	Cohen-Macaulay. 
	\end{cor}			
	\begin{proof}	
		 Since $e_1(\fkq, R) \leq 0$, We get by Theorem \ref{proe2c} that
		 $$e_1(\fkq :_R \fkm, R) \leq \mathcal{N}(R)$$ for all $C$-parameter ideals $\fkq$ of $R$.
	 
		Assume that $R$ is  Cohen-Macaulay. We get by Lemma \ref{fact2c} that $e_1(\fkq :_R \fkm, R) = \mathcal{N}(R)$ for all $C$-parameter ideals $\fkq$ of $R$. conversely, assume that $\fkq$ is a $C$-parameter ideals of $R$ such that 	$e_1(\fkq :_R \fkm, R) = \mathcal{N}(R)$. 
		Since $e_1(\fkq, R) \leq 0$, it follows from Theorem \ref{proe2c} that  
		$$\mathcal{N}(R) \geq e_1(\fkq :_R \fkm, R) - e_1(\fkq, R) \geq  = e_1(\fkq :_R \fkm, R) = \mathcal{N}(R).$$
		Thereore $\mathcal{N}(R) = e_1(\fkq :_R \fkm, R) - e_1(\fkq, R)$. Thus $R$ is Cohen-Macaulay because of Theorem \ref{proe2c}. This complete the proof 
	\end{proof}

	\begin{cor} \label{cor2}
		Assume that $\dim R \ge 2$. 
		Then  we have  
		$$ \mathcal{N}(R/\fku) \geq f_0(\fkq, R) \geq e_1(\fkq:_R \fkm, R) - e_1(\fkq, R) \geq r_d(R),$$
		for all $C$-parameter  ideals $\fkq\subset\m^{\rmg(R)}$  of $R$.
	\end{cor}
	\begin{proof} Put $S=R/\fku$, $I_R= \fkq:_R \m$ and  $I_S =  \fkq S:_S \m S$. Then  it follows from Corollary 3.2 in \cite{Sch99} that $S$ is a unmixed ring of  dimension $d$. 
		%\begin{claim}
		%$ e_1(I_R, R) - e_1(\fkq, R) \leq e_1(I_S, S) - e_1(\fkq, S)$
		%\end{claim}
		%\begin{proof}
		%Since $I_R \subseteq I_S$, 
		%$ \ell(S/I_S^{n+1}) \leq \ell(S/I_R^{n+1}S)$, and so $e_1(I_S, S) \geq e_1(I_R, S)$ because $I_R$ is integral over $\fkq$.
		%Thus, we get by Lemma \ref{lm3.4} that
		%$$ e_1(I_R, R) - e_1(\fkq, R) = e_1(I_R, S) - e_1(\fkq, S) \leq e_1(I_S, S) - e_1(\fkq, S).$$
		%\end{proof}
		Then, applying Theorem \ref{proe2c} we have $ \mathcal{N}(S) \geq f_0(\fkq, S)$. %\geq e_1(I_S, S) - e_1(\fkq, S) \geq e_1(I_R, R) - e_1(\fkq, R).$$
		In addition, it follows from Lemma 2.1 in \cite{TTr 20}, Lemma \ref{5.1} and Corollary 4.4 in \cite{OTY20} that
		$$  f_0(\fkq, S) \geq f_0(\fkq, R) \geq e_1(I_R, R) - e_1(\fkq, R) \geq r_d(R).$$
		Consequently, 
		$$ \mathcal{N}(S) \geq f_0(\fkq, R) \geq e_1(\fkq:_R \fkm, R) - e_1(\fkq, R) \geq r_d(R),$$
		and this complete the proof.
	\end{proof}

The following consequence of Theorem \ref{proe2c} provides a characterization of Gorenstein rings.
	\begin{cor} \label{cor3}
		Assume that $R$ is unmixed of dimension $d \geq 2$. Then the following
		statements are equivalent.
		\begin{enumerate}[$(1)$]
			\item   $R$ is Gorenstein.
			%\item[$(ii)$]  For all parameter ideals $\fkq \subseteq \m^2$, we have
			%$$ r_d(R) = e_1(\fkq:\m) - e_1(\fkq) = f_0(\fkq) = \mathcal{N}(R) = 1.$$
			\item For some $C$-parameter ideal $\fkq$ of $R$, we have
			$ \mathcal{N}(R) = 1.$
			\item For some $C$-parameter ideal $\fkq$ of $R$, we have
			$ f_0(\fkq, R) =1.$
			\item For some $C$-parameter ideal $\fkq\subseteq \fkm^{\rmg(R)}$, we have
			$ e_1(\fkq:\m) - e_1(\fkq) = 1.$
		\end{enumerate}
	\end{cor}
	\begin{proof}
		$ (1)\Rightarrow (2)$. Let $\fkq$ be a $C$-parameter ideal of $R$. Since $R$ is Gorenstein, 
		$$ \mathcal{N}(R) =  \ell((0):_{H_{\fkm}^{s}(R)} \fkm) = 1.$$ 
		$(2) \Rightarrow (3) $ and $ (3) \Rightarrow (4)$. Since $R$ is Gorenstein, $r_d(R) = \mathcal{N}(R) = 1$. Therefore by Corollary \ref{cor2}, we have  
		$$ \mathcal{N}(R) = f_0(\fkq, R) = \e_1(\fkq:_R\fkm) -\e_1(\fkq) = r_d(R) = 1.$$
		$(4) \Rightarrow (1)$.	
		Let $\fkq$  be a $C$-parameter ideal such that $\e_1(\fkq:_R\fkm) -\e_1(\fkq) =1 $  and $\fkq\subseteq \m^{\rmg(R)}$, then we have
		$\e_1(\fkq:_R\fkm) - \e_1(\fkq) \le r_d(R).$
		By Theorem \ref{proe2c}, $R$ is Cohen-Macaulay. Therefore, we have
		$$1 =\e_1(\fkq:_R\fkm) - \e_1(\fkq) = r_d(R).$$
		Hence, $R$ is Gorenstein, as required.
	\end{proof}

	%\begin{cor} \label{cor3}
	%Assume that $M$ is unmixed with $\dim M\geq 2$. The following
	%statements are equivalent.
	%\begin{enumerate}[$i)$]
	%	\item   $R$ is Gorenstein.
	%	\item For some $C$-parameter ideals $\fkq$ of $R$ contained in $\fkm^{\rmg(R)}$, one has  the inequalities
	%	$$ e_1(\fkq:\m) - e_1(\fkq) = 1.$$
	%\end{enumerate}

	%\end{cor}

	\section{The finiteness of the set of Chern numbers}
	In this section, we analyse  the boundness of the values $e_1(\fkq:_R\fkm, R)$ for parameter ideals $\fkq$ of $R$ and deduced that the local cohomology modules 
	$\{H^i_\fkm(R)\}_{i<d}$ are finitely generated, once $R$ is unmixed. 

	For each $n$, $i \geq 1$, we put
	$$ \Omega^{n}_i(R) = \{ e_i(\fkq:_R\fkm, R) \mid  \fkq \text{ is a $C$-parameter ideal of  } R \text{ contained in } \m^{n} \}.$$
	Note that $\Omega_{i}^{t}(R) \neq \emptyset $  and $\Omega_{i}^{t}(R) \subseteq \Omega_{i}^{t'}(R) \subseteq \Xi_i(R)$, for all $t \geq t' \geq 1$. Moreover, we get by Lemma \ref{fact2c} that if $R$ is generalized Cohen-Macaulay then $\Xi_i(R)$  $\Omega^{n}_i(R)$ is finite for all $i \geq 1$ and $n \ge 1$. 	
	
	Let  $\underline{x} = x_1, x_2, \ldots , x_d$ be a $C$-system of parameters of $R$. Put $\fkq^{[n]} = (x_1^{n}, x_2^{n}, \ldots , x_d^{n})$ and $I^{[n]} = \fkq^{[n]} :_R \m$, for all $n\ge 1$. For each  $t  \geq 1$,  let 
	$$ \Omega_{\underline{x}, 1}^{t}(R) = \{ e_1(I^{[n]}, R) \mid n \geq t\}.$$ 
	Note that $\Omega_{\underline{x}, 1}^{t}(R) \neq \emptyset $  and $\Omega_{\underline{x}, 1}^{t}(R) \subseteq \Omega_{\underline{x}, 1}^{t'}(R)$ for all $t \geq t' \geq 1$.

	%%%%%%%%%%%%%%%%%%%%%%%%%%%%%%%%%%%%%%%%%%%%%%%%%%%%%%%%%%%%%%%%%%%%%%%%55%%%

	%%%%%%%%%%%%%%%%%%%%%%%%%%%%%%%%%%%%%%%%%%%%%%%%%%%%%%%%%%%%%%%%%
	\begin{lem}\label{lm2}
		Assume that $R$ is unmixed of dimension $d \geq 2$. Suppose that there exist a   $C$-system of parameters $\underline{x}= x_1, \ldots, x_d$  of $R$ and an integer $t$ such that the set 
		$\Omega_{\underline{x}, 1}^{t}(R)$ is finite. Let 
		$$h = \max \{ \left| X \right| \,  : \,  X \in \Omega_{\underline{x},1}^{t}(R) \} + \mathcal{N}(R).$$ 
		Then $\m^h H_{\m}^{i}(R) = 0$ for all $ i \neq d$.
	\end{lem}
	\begin{proof}
		We use induction on the dimension $d$ of $R$. Suppose that $d = 2$. Since $R$ is unmixed, $R$ is generalized Cohen-Macaulay, and so $\mathcal{N}(R) =  2r_1(R) +r_2(R)$. We chosse $n \geq t$. Put $\fkq = (x_1^n, x_2^n)$ and  $I = \fkq:_R\fkm$. Then, 
		since $x_1$, $x_2$ is a $C$-system of parameters of $R$, by Lemma \ref{fact2c} we have
		%$$ e_1(I, R) = - \ell(H_{\m}^{1}(R)) + r_1(R) +r_2(R).$$
		%Thus we get that
		\begin{eqnarray*}
			\ell(H_{\fkm}^{1}(R)) &=& -e_1(I, R) + r_1(R) +r_2(R)\\
			&\leq& \left| e_1(I, R)  \right| + \mathcal{N}(R)\le h.
		\end{eqnarray*} 
		Hence,  $\m^h H_{\m}^{1}(R) = 0$. 
		
		Suppose that $d \ge 3$ and our assertion holds true for  $d -1$.  %Assume that there exist an   $C$-system of parameters $\underline{x}= x_1, \ldots, x_d$  of $R$ and an integer $t$ such that the set  $\Omega_{\underline{x}, 1}^{t}(R)$ is finite.
		Let $n\ge t$ be an integer such that $x_1^nH^1_\fkm(R) = (0)$. Let $y_1 = x_1^{n}$ and  $A = R/(y_1)$. Then $\dim A=d-1$ and  by the fact \ref{fact1}, we have $\Ass A\subset \Assh A\cup\{\fkm\}$. It follows that $U_A(0) = H^0_\fkm(A)$. Let $B = A/H_{\m}^{0}(A)$. Then $B$ is unmixed.
		
		Now, by Lemma \ref{lmc}, $ \underline{x'}= x_2, x_3, \ldots, x_d$ is a $C$-system of parameters of $A$. Thus, $\underline{x'}$ is  also a $C$-system of parameters of $B$.
		On the other hand, it follows from Lemma 2.4 in \cite{CuT08} that there exists a positive integer $t'\geq t$ such that 
		$$[(x_2^{m}, x_3^{m}, \ldots, x_d^{m}) + H_{\m}^0(A)]:_A \m A = [(x_2^{m}, x_3^{m}, \ldots, x_d^{m}):_A \m A] + H_{\m}^0(A),$$
		for all $m \geq t'$. Therefore, we have the following.
		\begin{claim}
			$\Omega_{\underline{x'}, 1}^{t'}(B) \subseteq \Omega_{\underline{x}, 1}^{t}(R)$.
		\end{claim}
		\begin{proof}
			For each $m\ge t'$, we put $y_2 = x_2^{m}, \ldots, y_d = x_d^{m}$,  $\fkq = (y_1,y_2 \ldots y_d)$, $I = \fkq:_R\m$, $\fkq' = (y_2, \ldots,y_d)$.
			Since $\underline{x}$ is  a $C$-system of parameters of $R$,  it forms a  $d$-sequence of $R$. Thus $x_1$ is a superficial element of $R$ with respect to $I$ and $\fkq$. 
			Therefore, by Lemma \ref{lmsup}, we have
			$$e_1(I, R) = e_1(I', A),$$
			where  $I' = \fkq'A :_A \m A = (\fkq :_R \m)A$.
			It follows from  Lemma \ref{lm3.4}, $d\ge 3$ and $[\fkq' + H_{\m}^0(A)]:_A \m A = [\fkq':_A \m A] + H_{\m}^0(A)$ that we have
			$$e_1(I', A) =  e_1(\fkq'B :_B \m B, B).$$ 
			Consequently, $ e_1(\fkq'B :_B \m B, B) = e_1(I, R) \in \Omega_{\underline{x}, 1}^{t}(R)$. Hence $\Omega_{\underline{x'}, 1}^{t'}(B) \subseteq \Omega_{\underline{x}, 1}^{t}(R)$, as required.
		\end{proof}
		
		It follows from the above claim that the set $\Omega_{\underline{x'}, 1}^{t'}(B) $ is finite. Thus by the hypothesis inductive on $d$, we have 
		$ \m^{h'}H_{\m}^{i}(B) =0,$
		for all  $i \neq d-1,$
		where $h' = \max \{ \left| X \right| \,  : \, X \in \Omega_{\underline{x'}, 1}^{t'}(B)  \} + \mathcal{N}(B)$. 
		On the other hand,  by Lemma \ref{lmc}, we have $$\mathcal{N}(R) = \mathcal{N}(A) = \mathcal{N}(B) + r_0(A).$$ 
		Therefore $h' \leq h$, and so $ \fkm^{h}H_{\fkm}^{i}(B) = 0$, for all  $i \neq d-1$. Hence $ \fkm^{h}H_{\fkm}^{i}(A) = 0$, for all  $1\le i \le d-2$.
		
		Now since $R$ is unmixed,  $y_1$ is $R$-regular. Therefore it follows from the exact sequence
		$$\xymatrix{0\ar[r]&R\ar[r]^{y_1}&R\ar[r]&A\ar[r]&0}$$
		that we have the surjective maps
		$\xymatrix{H^{i}_\fkm(A)\ar@{->>}[r]&(0) :_{H^{i+1}_{\fkm}(R)} y_1}$
		for all $i\le d-2$ and the injective map
		$\xymatrix{H^{1}_\fkm(R)\ar@{^{(}->}[r]&H^{1}_{\fkm}(A),}$
		since $y_1H^1_{\fkm}(R)=0$. Thus we have
		$$ \fkm^{h}[ (0) :_{H^{i+1}_{\fkm}(R)} y_1  ] = (0) \text { and }  \fkm^{h}H^{1}_\fkm(R)=0$$
		for all $ 1\leq i  \leq d - 2$. It follows from $H^{i}_{\fkm} (R)=\bigcup\limits_{n\ge1} [(0) :_{H^{i}_{\fkm}(R)} \fkm^n] $ that we have $\fkm^{h}H^{i}_{\fkm} (R) = (0)$ for all $ i\neq d $, as required. 
	\end{proof}
	%%%%%%%%%%%%%%%%%%%%%%%%%%%%%%%%%%%%%%%%%%%%%%%%%%%%%%%%%%%%5
	\begin{lem}\label{lm5}
	Assume that $R$ is unmixed of dimension $d \geq 2$. Suppose that $R/\fku$ is generalized Cohen-Macaulay. Then
	$$ \fkq:_R \fkm = \fkq :_{R/\fku} \fkm,$$
for all $C$-parameter ideals $\fkq$  of $R/\fku$.	
\end{lem}
\begin{proof}
Put $S= R/\fku$. Let $\fkq$ is a $C$-parameter of $S$, $I_R = \fkq :_R \fkm$ and $I_S = \fkq :_{S} \fkm S$.   Then $\fkq$ is also a  parameter of $R$.	
	Since $S$ is generalized Cohen-Macaulay, we have $(I_S)^2=\fkq I_S$. Thus we get by Theorem 1.1 in \cite{GoN03} that
	$e_1(I_S, S)=\ell(S/I_S)-e_0(\fkq,S)+ e_1(\fkq, S)$.
	By Theorem 3.1 in \cite{GoN03}, we have
	\begin{equation*}\label{eq11}
		\begin{aligned}
			\ell(S/I_S)-e_0(\fkq,S)&= e_1(I_S, S) - e_1(\fkq, S)\\
			&\ge e_1(I_R, S) - e_1(\fkq, S)\ge \ell(S/I_R S)-e_0(\fkq,S).
		\end{aligned}
	\end{equation*}
	Then we have $I_S=I_R$ since $I_R \subseteq I_S$.

\end{proof}

%%%%%%%%%%%%%%%%%%%%%%%%%%%%%%%%%%%%%%%%%%%%%%%%%%%%%%%%%%%%%%%%%	

	\begin{thm}\label{prop2}
		Assume that  $ d  \geq 2$. Then the following statements are equivalent.
		\begin{enumerate}
			\item[$(1)$] $R/\fku$ is generalized Cohen-Macaulay and $\dim \fku\le d-2$.
			\item[$(2)$] There exists an integer $t$ such that the set $\Omega^{t}_1(R)$ is finite.
		\end{enumerate}
	\end{thm}
\begin{proof}
	$(1) \Rightarrow (2)$. Let $\fkq$ be a $C$-parameter  ideals of $R$ containing in $\m^{\rmg(R)}$.  We get by Corollary \ref{cor2} that 
$$ \mathcal{N}(R/\fku) \geq  e_1(\fkq:_R \fkm, R) - e_1(\fkq, R) \geq r_d(R).$$
 On the other hand, the set $\{ e_1(\fkq, R) \mid \fkq \text{ is a parameter of } R \}$ is finite because of Theorem 4.5 in \cite{GGH+14}. Thus $\Omega^{t}_1(R)$ is finite for all $t \geq {\rmg(R)}$.\\
$(2) \Rightarrow (1)$. %Let $\fku$ be the unmixed component of the ideal $(0)$ in $R$. 
We put $S = R/\fku$ and $u = \dim \fku$. 
 	Then there exists  a $C$-system  $\underline x=x_1, \ldots, x_d$ of parameters  of $R$ such that 
\begin{enumerate}[$i)$]
	\item
	$(\underline x)\subseteq \fkm^t$, 
	\item $(x_{u+1}, x_{u+2}, \ldots, x_d)\fku = 0$,
	\item  $\underline x$ is a $C$-system of parameters of $S$.
\end{enumerate}
Let $\fkq^{[n]} = (x_1^{n}, \ldots, x_d^{n})R$, $I_R^{[n]}= \fkq^{[n]}:_R \m$ and  $I_S^{[n]} =  \fkq^{[n]}S:_S \m S$, where $n \geq t$. 
%	\begin{claim} \label{cl1}
%There exists an integer $t' \geq t$ such that the set $\Omega_{\underline{x}, 1}^{t'}(S)$ is finite.
%\end{claim}
%\begin{proof}
We have $I_R^{[n]} \subseteq I_S^{[n]}$, and so we get
$ \ell(S/(I_S^{[n]})^{m+1}) \leq \ell(S/(I_R^{[n]})^{m+1}S)$ for all $m\ge 0$. Thus $e_1(I_S^{[n]}, S) \geq e_1(I_R^{[n]}, S)$, because $I_S^{[n]}$ is integral over $\fkq^{[n]} S$ by Proposition 2.3 in \cite{GoS03}.
Therefore by Corollary \ref{cor2}, we have 
\begin{equation}\label{eq1}
	\begin{aligned}
		\mathcal{N}(S)\ge e_1(I_S^{[n]}, S) - e_1(\fkq^{[n]}, S)&\ge e_1(I_R^{[n]}, S) - e_1(\fkq^{[n]}, S)&\\
		&= e_1(I_R^{[n]}, R) - e_1(\fkq^{[n]}, R) \ge r_d(R)=r_d(S).
	\end{aligned}
\end{equation}
Since the set $\Omega^{t}_{1}(R)$ is finite, so is the set $ \{e_1(I_R^{[n]}, R)\mid n\ge0\} $. Thus the set $ \{e_1(\fkq_R^{[n]}, R)\mid n\ge0\} $ is finite.  
By Lemma \ref{lm3.4}, we have 
$$0 \geq  e_1(\fkq_R^{[n]}, S) \geq  e_1(\fkq_R^{[n]}, R).$$
Therefore, the set $ \{e_1(\fkq_R^{[n]}, S)\mid n\ge0\} $ is finite.  By (\ref{eq1}),  the set $\Omega_{\underline{x}, 1}^{t'}(S)$ is finite for some $t'$. 
%\end{proof}	
It follows from $S$ is unmixed and Lemma \ref{lm2} that $S$ is generalized Cohen-Macaulay.	

Now suppose that $u = d-1$. It follows from Lemma \ref{fact2c} %Theorem 5.2 in \cite{CQT19} 
that 
$$  e_{1}(I_S^{[n]} , S) = -\left( \sum_{j=1}^{d-1} \binom{d-2}{j-1} \ell(H_{\fkm}^{j}(S)) - \sum_{j=1}^{d} r_j(S) \right),$$
which is independent of the choice of $n$.
On the other hand, we have
$$  e_0(I_R^{[n]}, \fku) = e_0(\fkq^{[n]}, \fku) = e_0((x_1^{n}, \ldots, x_u^{n}), \fku) = n^u e_0(\fkq, \fku ) \geq n^t.$$
Thus, by Lemma \ref{lm3.4}, we get that
\begin{eqnarray*}
	-e_{1}(I_R^{[n]}, R) &=& -e_{1}(I_R^{[n]}, S) + e_0(I_R^{[n]}, \fku)  \\
	&=& -e_{1}(I_S^{[n]}, S) + n^u e_0(\fkq, \fku ),
\end{eqnarray*} 
which is in  contradiction with the finiteness of $\Omega^{n}_{1}(R)$ and $\Omega^{n}_{1}(S)$. Hence $u \leq d-2$.

\end{proof}	
	
		%%%%%%%%%%%%%%%%%%%%%%%%%%%%%%%%%%%%%%%%%%%%%%%%%%%%%%%%%%%%%%%%%%%%%%% 
	Applying Theorem \ref{prop2} we obtain the following result.	
	\begin{thm}\label{prop1}
	Assume that $R$ is unmixed with $ \dim R  \geq 2$. Then the following statements are equivalent.
	\begin{enumerate}
		\item[$(1)$] $R$ is generalized Cohen-Macaulay.
		\item[$(2)$] The set $\Xi_1(R)$ is finite.
	\end{enumerate}
\end{thm}	
	\begin{proof}
	It is immediate from Theorem  \ref{prop2}.
\end{proof}

	\begin{thm}\label{thm2}
		Assume that  $ d  \geq 2$. Then the following statements are equivalent.
		\begin{enumerate}
			\item[$1)$]   $ R/H_{\m}^0(R)$ is Cohen-Macaulay.
			\item[$2)$] The sets $\Xi_i(R)$ are finite for all $i \ge 2$ and there exists a  $C$-parameter ideal  $\fkq\subseteq\fkm^{\rmg(R)}$ of  $R$ such that  $$e_1(\fkq:_R\fkm, R) - e_1(\fkq, R) \leq r_d(R).$$
			
		\end{enumerate}
	\end{thm}
	\begin{proof} Let  $S=R/\fku$.\\
		$1) \Rightarrow 2)$.   Since $ S$ is Cohen-Macaulay, $R$ is  generalized Cohen-Macaulay. Therefore $\Xi_i(R)$ are finite for all $i \ge 2$. Let $\fkq\subseteq\fkm^{\rmg(R)}$ be a $C$-parameter ideals of $R$. Then by Lemma 2.4 in \cite{CuT08}, we can assume that
		$$ [\fku + \fkq] : \m = \fku + [\fkq : \m].$$ 
		It follows from the Lemma \ref{fact2c}, Lemma \ref{lm3.4} that  
		$$e_1(\fkq:_R \fkm, R) - e_1(\fkq, R) =  e_1(\fkq:_S \fkm, S) - e_1(\fkq, S)=  e_1(\fkq S:_S \fkm S, S) - e_1(\fkq, S) = r_d(S) = r_d(R),$$
		as required.
		
		$2) \Rightarrow 1)$. %Let $I_R= \fkq:_R \m$ and  $I_S =  \fkq S:_S \m S$. 
		%. Then  it follows from Corollary 3.2 in \cite{Sch} that $S$ is a unmixed ring of  dimension $d$. 
		%\begin{claim}
		%$ e_1(I_R, R) - e_1(\fkq, R) \leq e_1(I_S, S) - e_1(\fkq, S)$
		%\end{claim}
		%\begin{proof}
		%Since $I_R \subseteq I_S$, we have
		%$ \ell(S/I_S^{n+1}) \leq \ell(S/I_R^{n+1}S)$, and so $e_1(I_S, S) \geq e_1(I_R, S)$, because $I_S$ is integral over $\fkq S$.
		%Then by Lemma \ref{lm3.4}, we have 
		%$$\begin{aligned}
		%e_1(I_S, S) - e_1(\fkq, S)&\le e_1(I_R, S) - e_1(\fkq, S)\\
		%&=e_1(I_R, R) - e_1(\fkq, R) \le r_d(R)=r_d(S).
		%\end{aligned}$$  
		Since  $$e_1(\fkq:_R\fkm, R) - e_1(\fkq, R) \leq r_d(R)$$
		for some $C$-parameter ideals $\fkq\subseteq\fkm^{\rmg(R)}$ of  $R$, by Lemma \ref{prof0c}, $S$ is Cohen-Macaulay.

		Suppose that $u=\dim \fku\geq 1$. 
		%Then, it is enough to show that $S$ is generalized Cohen-Macaulay ring and $u = 0$. 
		%Assume that  $\Omega^{t}_i(R)$ are finite for all $ i = 1, 2, \ldots, d$.  
		Then there exists  a $C$-system  $\underline x=x_1, \ldots, x_d$ of parameters  of $R$ such that 
		\begin{enumerate}[$i)$]
			\item
			$(\underline x)\subseteq \fkm^t$, 
			\item $(x_{u+1}, x_{u+2}, \ldots, x_d)\fku = 0$,
			\item  $\underline x$ is a $C$-system of parameters of $S$.
		\end{enumerate}
		%Then there exists  a system  $Q=(x_1, \ldots, x_d)$ of parameters of $R$ such that 
		%$Q\subseteq \fkm^t$ and $(x_{u+1}, x_{u+2}, \ldots, x_d)\fku = 0$.
		Let $\fkq^{[n]} = (x_1^{n}, \ldots, x_d^{n})R$, $I_R^{[n]}= \fkq^{[n]}:_R \m$ and  $I_S^{[n]} =  \fkq^{[n]}S:_S \m S$ for all $n\ge 0$. 
%		Then, since $S$ is Cohen-Maculay, by  Lemma \ref{lm4} we have $ [W +  Q^{[n]}] :_R \m = W + [Q^{[n]}  :_R \m],$ and $\ir_{S}(Q^{[n]}) = r_d(S)$  for all $n\ge 0$  and so $J_R^{[n]} =J_S^{[n]}$.
		Since $S$ is Cohen-Macaulay, we get by Lemma \ref{lm5} that %$(I_S^{[n]})^2=\fkq^{[n]} I_S^{[n]}$. Thus we have
		%$e_1(I_S^{[n]}, S)=\ell(S/I_S^{[n]})-e_0(Q,S)+ e_1(Q^{[n]}, S)$ and $e_i(I_S^{[n]}, S)=e_{i-1}(Q^{[n]}, S)+e_i(Q^{[n]}, S)$ for all $i=2,\ldots, d$.
		%By Theorem 3.1 in \cite{GoN03}, we have
		%\begin{equation}\label{eq11}
		%	\begin{aligned}
		%		\ell(S/I_S^{[n]})-e_0(Q,S)&= e_1(I_S^{[n]}, S) - e_1(Q^{[n]}, S)\\
		%		&\ge e_1(I_R^{[n]}, S) - e_1(Q^{[n]}, S)\ge \ell(S/I_R^{[n]}S)-e_0(Q,S).
		%	\end{aligned}
		%\end{equation}
		%Then we have 
		$I_S^{[n]}=I_R^{[n]}$ 
		% since $I_R^{[n]} \subseteq I_S^{[n]}$.
		Moreover, it follows Lemma \ref{fact2c} %Theorem 5.2 in \cite{CQT19} 
		that 
		$$  e_{d-u}(I_S^{[n]} , S) = 
		\begin{cases}
			0       &\text{ if } u \leq d-2,\\
			r_d (S)    &\text{ if } u = d-1.
		\end{cases}$$
		which is independent of the choice of $n$.
		On the other hand, we have
		$$  e_0(I_R^{[n]}, \fku) = e_0(\fkq^{[n]}, \fku) = e_0((x_1^{n}, \ldots, x_u^{n}), \fku) = n^u e_0(\fkq, \fku ) \geq n^u.$$
		Thus, by Lemma \ref{lm3.4}, we get that
		\begin{eqnarray*}
			(-1)^{d-u}e_{d-u}(I_R^{[n]}, R) &=& (-1)^{d-u}e_{d-u}(I_R^{[n]}, S) + e_0(I_R^{[n]}, \fku)  \\
			&=& (-1)^{d-u}e_{d-u}(I_S^{[n]}, S) + n^u e_0(\fkq, \fku )\\
			&\geq & n^u+ r_{d}(S),
		\end{eqnarray*} 
		which is in contradiction with the finiteness of $\Xi_{d-u}(R)$. Hence $u = 0$, and this complete the proof. 
	\end{proof}

	\begin{thm}\label{thm3}
		Assume that $\dim R  \geq 2$. Then the following statements are equivalent.
		\begin{enumerate}
			\item[$(i)$] $R$ is generalized Cohen-Macaulay.
			\item[$(ii)$] The sets $\Xi_i(R)$ are finite for all $ i = 1, 2, \ldots, d$.
		\end{enumerate}
	\end{thm}
	\begin{proof}
		$(i) \Rightarrow (ii)$ follows from Lemma \ref{fact2c}.\\
		$(ii) \Rightarrow (i)$. %Let $\fku$ be the unmixed component of the ideal $(0)$ in $R$. 
		We put $S = R/\fku$ and $u = \dim \fku$. Assume that $u >0$. Since $\Xi_1(R)$ is finite, it follows from Theorem \ref{prop1} that $S$ is generalized Cohen-Macaulay.
		%Then, it is enough to show that $S$ is generalized Cohen-Macaulay ring and $u = 0$. 
		%Assume that  $\Omega^{t}_i(R)$ are finite for all $ i = 1, 2, \ldots, d$.  
		Then there exists  a $C$-system $\underline x=x_1, \ldots, x_d$  of parameters  of $R$ such that 
		\begin{enumerate}[$i)$]
			\item
			$(\underline x)\subseteq \fkm^t$, 
			\item $(x_{u+1}, x_{u+2}, \ldots, x_d)\fku = 0$,
			\item  $\underline x$ is a $C$-system of parameters of $S$.
		\end{enumerate}
		Let $\fkq^{[n]} = (x_1^{n}, \ldots, x_d^{n})R$, $I_R^{[n]}= \fkq^{[n]}:_R \m$ and  $I_S^{[n]} =  \fkq^{[n]}S:_S \m S$ for all $n\ge 0$. 	Since $S$ is generalized Cohen-Macaulay, we get by Lemma \ref{lm5} that 
		$I_S^{[n]}=I_R^{[n]}$. 
		Moreover, it follows Lemma \ref{fact2c} %Theorem 5.2 in \cite{CQT19} 
		that 
		$$  e_{d-u}(I_S^{[n]} , S) = (-1)^{d-u} \left( \displaystyle\sum_{j=1}^{u} \binom{u-1}{j-1} \ell_{R}(H_{\m}^{j}(S)) - \sum_{j=1}^{u+1} r_j(S) \right),$$
		which is independent of the integer $n$.
		On the other hand, we have
		$$  e_0(I_R^{[n]}, \fku) = e_0(\fkq^{[n]}, \fku) = e_0((x_1^{n}, \ldots, x_u^{n}), \fku) = n^u e_0(\fkq, \fku ) \geq n^u.$$
		Thus, by Lemma \ref{lm3.4}, we get that
		\begin{eqnarray*}
			(-1)^{d-u}e_{d-u}(I_R^{[n]}, R) &=& (-1)^{d-u}e_{d-u}(I_R^{[n]}, S) + e_0(I_R^{[n]}, \fku)  \\
			&=& (-1)^{d-u}e_{d-u}(I_S^{[n]}, S) + n^u e_0(\fkq, \fku ),
		\end{eqnarray*} 
		which is in  contradiction with the finiteness of $\Xi_{d-u}(R)$ and $\Xi_{d-u}(S)$. Hence $u = 0$, as required. 
	\end{proof}
\begin{rem}
 In \cite{KoT15}, A. Koura and N. Taniguchi proved that the Chern numbers of $\fkm$-primary ideals in $R$ range among only finitely many values if and only if $\dim  R = 1$ and $R/H_{\m}^0(R)$ is analytically unramified. 
	\end{rem}	
Let us note the following example of non-generated Cohen-Macaulay local rings $R$ with $\left| \Xi_1(R) \right| = 1$. Moreover, this example also shows that  the statements $(1), (4)$ and   $(5)$ in Theorem \ref{proe2c} will not be equivalent, if the unmixed condition is removed from the hypothesis.

	\begin{ex}
		Let $S = k[[X, Y, Z, W ]]$ be the
		formal power series ring over a field $k$. Put
		$R = S/[(X, Y, Z) \cap ( W )]$. Then 
		\begin{enumerate}[$(1)$] 
		\item $R$ is not unmixed with $\dim R = 3$. Moreover, $R$ is not generalized Cohen-Macaulay.
		\item $\left| \Xi_1(R) \right| = 1$.
		\item  For all  parameter ideals $\fkq$ in $R$, we have $$f(\fkq, R) = e_1(\fkq:_R\m) - e_1(\fkq) = r_d(R).$$
				\end{enumerate}
	       \begin{proof}
	        		We put $A = S/(W)$ and $B =
		S/(X, Y, Z)$. Then $A$ and $B$ are  Cohen-Macaulay. Then $R$ is not generalized Cohen-Macaulay because $H_{\fkm}^1(R) = H^1_\fkm(B)$ is not a finitely generated $R$-module.
		Let $\fkq$ be a parameter ideal in $R$ and put $I = \fkq : \fkm$. Then $I^n = \fkq^n : \fkm$, for all $n\geq 0$.  Thanks to the exact sequence $0\to B \to R \to A \to 0$,
		the sequence
		$$\xymatrix{0\ar[r]&B/\fkq^{n+1}B\ar[r]&R/\fkq^{n+1}R\ar[r]&A/\fkq^{n+1}A\ar[r]&0}$$
		is exact. By applying the functor
		$\mathrm{Hom}_R(R/\fkm, \bullet)$ we obtain the following exact sequence
		$$\xymatrix{0\ar[r]&[\fkq^{n+1}:_B\fkm]/\fkq^{n+1}\ar[r]&[\fkq^{n+1}:_R\fkm]/\fkq^{n+1}\ar[r]&[\fkq^{n+1}:_A\fkm]/\fkq^{n+1}\ar[r]&0}.$$
		Therefore, we have
		$$\begin{aligned}
			\ell_R(R/\fkq^{n+1})&=\ell_R(A/\fkq^{n+1})+\ell_R(B/\fkq^{n+1})\\
			&=\ell_R(A/\fkq A)\binom{n+3}{3} + \ell_R(B/\fkq B) \binom{n+1}{1}
		\end{aligned}$$
		and
		$$\begin{aligned}
			\ell_R([\fkq^{n+1}:_R\fkm]/\fkq^{n+1})&=\ell_R([\fkq^{n+1}:_A\fkm]/\fkq^{n+1})+\ell_R([\fkq^{n+1}:_B\fkm]/\fkq^{n+1})\\
			&=\binom{n+1}{2}+ 1,
		\end{aligned}$$
		for all integers $n \geqslant 0$.  Since $\ell_R(R/I^{n+1}) =  \ell_R(R/\fkq^{n+1}) - \ell_R([\fkq^{n+1}:_R\fkm]/\fkq^{n+1})$, we have 
		$e_1(\fkq:_R\m) = 0 + 1 = 1$. Thus $\left| \Xi_1(R) \right| = 1$.
		Moreover, we have 
		$ f(\fkq, R) = e_1(\fkq:_R\m) - e_1(\fkq) = r_d(R) = 1,$
		as required.
		\end{proof}
	\end{ex}
	
Now we close this paper with the following open question, which are suggested during the work in this paper, on the characterization of the Cohen-Macaulayness  in terms of its Chern numbers and stable value as follows.
\begin{que}\rm
	Assume that $\dim R\ge 2$. Then is $R$ is Cohen-Macaulay if and only if we have $ \mathcal{N}(R) \leq e_1(\fkq:\m) - e_1(\fkq)$  for some $C$-parameter ideal $\fkq$ of $R$?
%\begin{enumerate}[$(1)$] 
%	\item   $$\Xi_{1}(R) \subseteq (- \infty, \mathcal{N}(R)].$$ 
%	Moreover, $\mathcal{N}(R) \in \Xi_{1}(R)$ if and only if $R$ is 	Cohen-Macaulay. 
%	\item or $ \mathcal{N}(R) \leq f_0(\fkq, R).$
%\end{enumerate}	 
\end{que} 	
	
%{\bf Acknowledgments:}  The authors would like to thank K. Ozeki for the valuable comments to improve this article.
	


\begin{bibdiv}
\begin{biblist}

\bib{CuC18}{article}{
      author={Cuong, N.T.},
      author={Cuong, D.T.},
       title={{Local cohomology annihilators and Macaulayfication}},
        date={2018},
     journal={Acta Mathematica Vietnamica},
      volume={42},
       pages={37\ndash 60},
}

\bib{CGT13}{article}{
      author={Cuong, N.~T.},
      author={Goto, S.},
      author={Truong, H.~L.},
       title={Hilbert coefficients and sequentially {C}ohen--{M}acaulay
  modules},
        date={2013},
     journal={J. Pure Appl. Algebra},
      volume={217},
       pages={470\ndash 480},
}

\bib{CuQ17}{article}{
      author={Cuong, N.~T.},
      author={Quy, P.~H.},
       title={On the structure of finitely generated modules over quotients of
  cohen-macaulay rings},
        date={2017},
     journal={preprint, arXiv:1612.07638v3},
}

\bib{CuQ19}{article}{
      author={Cuong, N.~T.},
      author={Quy, P.~H.},
       title={On the index of reducibility of parameter ideals: the stable and
  limit values},
        date={2020},
     journal={Acta Mathematica Vietnamica},
      volume={45},
       pages={249\ndash 260},
}

\bib{CuQ20}{article}{
      author={Cuong, N.T.},
      author={Quy, P.H.},
       title={{On the Index of Reducibility of Parameter Ideals: the Stable and
  Limit Values}},
        date={2020},
     journal={Acta Mathematica Vietnamica},
      volume={45},
       pages={249\ndash 260},
}

\bib{CQT15}{article}{
      author={Cuong, N.~T.},
      author={Quy, P.~H.},
      author={Truong, H.~L.},
       title={On the index of reducibility in noetherian modules},
        date={2015},
     journal={J. Pure Appl. Algebra},
      volume={219},
       pages={4510\ndash 4520},
}

\bib{CQT19}{article}{
      author={Cuong, N.~T.},
      author={Quy, P.~H.},
      author={Truong, H.~L.},
       title={The index of reducibility of powers of a standard parameter
  ideal},
        date={2019},
     journal={Journal of Algebra and Its Applications},
      volume={18},
       pages={1950048, 17 pp,},
}

\bib{CST78}{article}{
      author={Cuong, N.T.},
      author={Schenzel, P.},
      author={Trung, N.V.},
       title={{Verallgemeinerte Cohen--Macaulay-Moduln}},
        date={1978},
     journal={Math. Nachr.},
      volume={85},
       pages={57\ndash 73},
}

\bib{CuT08}{article}{
      author={Cuong, N.~T.},
      author={Truong, H.~L.},
       title={Asymptotic behaviour of parameter ideals in generalized
  cohen-macaulay module},
        date={2008},
     journal={J. Algebra},
      volume={320},
       pages={158\ndash 168},
}

\bib{GGH+10}{article}{
      author={Ghezzi, L.},
      author={Goto, S.},
      author={Hong, J.},
      author={Ozeki, K.},
      author={Phuong, T.~T.},
      author={Vasconcelos, W.~V.},
       title={{Cohen-Macaulayness versus the vanishing of the first Hilbert
  coefficient of parameter ideals}},
        date={2010},
     journal={J. Lond. Math. Soc},
      volume={81},
       pages={679\ndash 695},
}

\bib{GGH+14}{article}{
      author={Ghezzi, L.},
      author={Goto, S.},
      author={Hong, J.},
      author={Ozeki, K.},
      author={Phuong, T.~T.},
      author={Vasconcelos, W.~V.},
       title={The chern numbers and euler characteristics of modules},
        date={2014},
     journal={Acta Mathematica Vietnamica},
      volume={40(1)},
       pages={37\ndash 60},
}

\bib{GoN03}{article}{
      author={Goto, S.},
      author={Nishida, K.},
       title={Hilbert coefficients and buchsbaumness of associated graded
  rings},
        date={2003},
     journal={J. Pure Appl. Algebra},
      volume={181},
       pages={61\ndash 74},
}

\bib{GoO11}{article}{
      author={Goto, S.},
      author={Ozeki, K.},
       title={Uniform bounds for hilbert coefficients of parameters},
        date={2011},
     journal={Commutative Algebra and its Connections to Geometry, Am. Math.
  Soc., Providence, RI},
      volume={555},
       pages={97\ndash 118},
}

\bib{GoS03}{article}{
      author={Goto, S.},
      author={Sakurai, H.},
       title={The equality {$I^{2} = QI$} in {B}uchsbaum rings},
        date={2003},
     journal={Rend. Semin. Mat. Univ. Padova},
      volume={110},
       pages={25\ndash 56},
}

\bib{GoS84}{article}{
      author={Goto, S.},
      author={Suzuki, N.},
       title={Index of reducibility of parameter ideals in a local ring},
        date={1984},
     journal={J. Algebra},
      volume={87},
       pages={53\ndash 88},
}

\bib{OTY20}{article}{
      author={K.~Ozeki, H. L.~Truong},
      author={Yen, H.~N.},
       title={On hilbert coefficients and sequentially cohen-macaulay rings},
        date={2020},
     journal={Preprint, arXiv:2103.11334},
}

\bib{KoT15}{article}{
      author={Koura, Asuki},
      author={Taniguchi, Naoki},
       title={{Bounds for the First Hilbert Coefficients of $\fkm$-primary
  Ideals}},
        date={2015},
     journal={Tokyo J. Math.},
      volume={38},
      number={1},
       pages={125\ndash 133},
}

\bib{SMV10}{article}{
      author={M.~Mandal, B.~Singh},
      author={Verma, J.~K.},
       title={On some conjectures about the chern numbers of filtrations},
        date={2011},
     journal={J. Algebra},
      volume={325},
       pages={147\ndash 162},
}

\bib{MoQ17}{article}{
      author={Morales, M.},
      author={Quy, P.~H.},
       title={{A Study of the Length Function of Generalized Fractions of
  Modules}},
        date={2017},
     journal={Proceedings of the Edinburgh Mathematical Society},
      volume={60},
      number={3},
       pages={721\ndash 737},
}

\bib{Nag62}{book}{
      author={Nagata, M.},
       title={Local rings},
   publisher={Interscience New York},
        date={1962},
}

\bib{Noe21}{article}{
      author={Noether, E.},
       title={Idealtheorie in ringbereichen},
        date={1921},
     journal={Math. Ann.},
      volume={83},
       pages={24\ndash 66},
}

\bib{Sam51}{article}{
      author={Samuel, P.},
       title={{La notion de multiplicit{\'e} en alg{\`e}bre et en
  g{\'e}om{\'e}trie alg{\'e}brique (French),}},
        date={1951},
     journal={J. Math. Pures Appl.},
      volume={30},
       pages={159\ndash 205},
}

\bib{Sch99}{article}{
      author={Schenzel, P.},
       title={On the dimension filtration and {Cohen-Macaulay} filtered
  modules},
        date={1999},
     journal={Van Oystaeyen, Freddy (ed.), Commutative algebra and algebraic
  geometry, New York: Marcel Dekker. Lect. Notes Pure Appl. Math.},
      volume={206},
       pages={245\ndash 264},
}

\bib{Tru14}{article}{
      author={Truong, H.~L.},
       title={Index of reducibility of parameter ideals and {C}ohen-{M}acaulay
  rings},
        date={2014},
     journal={J. Algebra},
      volume={415},
       pages={35\ndash 49},
}

\bib{Tru17}{article}{
      author={Truong, H.~L.},
       title={Chern coefficients and {C}ohen-{M}acaulay rings},
        date={2017},
     journal={J. Algebra},
      volume={490},
       pages={316\ndash 329},
}

\bib{Tru86}{article}{
      author={Trung, N.V.},
       title={{Toward a theory of generalized Cohen-Macaulay modules}},
        date={1986},
     journal={Nagoya Math. J},
      volume={102},
       pages={1\ndash 49},
}

\bib{TTr20}{article}{
      author={Tam, N. T.~T.},
      author={Truong, H.~L.},
       title={A note on chern coefficients and cohen-macaulay rings},
        date={2020},
     journal={Arkiv f{\"o}r Matematik},
      volume={58, No. 1},
       pages={197\ndash 212},
}

\bib{Vas08}{article}{
      author={Vasconcelos, W.~V.},
       title={The chern cofficient of local rings},
        date={2008},
     journal={Michigan Math},
      volume={57},
       pages={725\ndash 713},
}

\end{biblist}
\end{bibdiv}
\end{document}